\documentclass{daj}

\usepackage{amssymb}
\usepackage{amsthm}
\usepackage{amsmath}
\usepackage[capitalise]{cleveref}

\newtheorem{prop}{Proposition}[section]
\newtheorem{theorem}[prop]{Theorem}
\newtheorem{lemma}[prop]{Lemma}
\newtheorem{corollary}[prop]{Corollary}
\newtheorem{conjecture}[prop]{Conjecture}

\newcommand{\N}{\mathbb{N}}
\newcommand{\R}{\mathbb{R}}
\newcommand{\G}{\Gamma}
\newcommand{\eps}{\varepsilon}
\newcommand{\Aut}{\text{\textup{Aut}}\,}

\dajAUTHORdetails{%
  title = {Small Doubling Implies Small Tripling for Balls of Large Radius}, 
  author = {Romain Tessera and Matthew Tointon},
  plaintextauthor = {Romain Tessera, Matthew Tointon},
}   

\dajEDITORdetails{%
   year={2025},
   number={9},
   received={4 December 2024},   
   published={2 September 2025},  
   doi={10.19086/da.143426},       
}   

\begin{document}

\begin{frontmatter}[classification=text]


\author[romain]{Romain Tessera 
\thanks{Supported by ANR-24-CE40-3137 }}
\author[matt]{Matthew Tointon}

\begin{abstract}
We show that if $K\ge1$ is a parameter and $S$ is a finite symmetric subset of a group containing the identity such $|S^{2n}|\le K|S^n|$ for some integer $n\ge2K^2$, then $|S^{3n}|\le\exp(\exp(O(K^2)))|S^n|$. Such a result was previously known only under the stronger assumption that $|S^{2n+1}|\le K|S^n|$. We prove similar results for locally compact groups and vertex-transitive graphs. We indicate some results in the structure theory of vertex-transitive graphs of polynomial growth whose hypotheses can be weakened as a result.
\end{abstract}
\end{frontmatter}

\section{Introduction}
The famous \emph{Plünnecke--Ruzsa} inequalities state that if $K\ge1$ is a parameter and $A$ is a finite subset of an abelian group satisfying the \emph{small doubling} hypothesis $|A+A|\le K|A|$, then $|mA-nA|\le K^{m+n}|A|$ for all non-negative integers $m,n$ \cite{plunnecke,ruz.pl.1,ruz.pl.2} (see also \cite{petridis} for a dramatically simplified proof). It has long been known that the directly analogous statement fails for arbitrary non-abelian groups, and that a bound of the form $|A^2|\le K|A|$ does not in general imply a bound of the form $|A^3|\le O_K(|A|)$ (see e.g. \cite[Example 2.5.2]{book}, which is essentially reproduced from \cite{tao.product.set}). Nonetheless, under the stronger \emph{small tripling} assumption $|A^3|\le K|A|$, we do have that $|A^{\epsilon_1}\cdots A^{\epsilon_m}|\le K^{3(m-2)}|A|$ for all choices of $\epsilon_i=\pm1$ and all $m\ge3$ \cite[Proposition 2.5.3]{book}. A similar result holds more generally if $A$ is an open precompact subset of a locally compact group, with Haar measure in place of cardinality \cite[Lemma 3.4]{tao.product.set}.

Breuillard and the second author showed that if $A$ is a \emph{ball} in a Cayley graph then the small-tripling hypothesis can \emph{almost} be weakened to a small-doubling hypothesis, as follows.
\begin{theorem}[Breuillard--Tointon {\cite[Lemma 2.10]{bt}}]\label{lem:bt}
Let $K\ge1$. Suppose $S$ is a finite symmetric subset of a group containing the identity, and $|S^{2n+1}|\le K|S^n|$ for some $n\in\N$. Then $|S^{3n}|\le O_K(|S^n|)$, and $S^{2n}$ is an $O_K(1)$-approximate group.
\end{theorem}
A \emph{$K$-approximate group} is a symmetric subset $A$ of a group containing the identity such that $A^2$ is covered by at most $K$ left translates of $A$. The last conclusion of \cref{lem:bt} is not stated in the reference, but follows immediately from the first conclusion and \cref{lem:tripling->AG}, below.

The purpose of the present note is to show that, for sufficiently large radii, a genuine small-doubling hypothesis is sufficient in \cref{lem:bt}. Indeed, we show this more generally for locally compact groups, in which context nothing weaker than small tripling was previously known to suffice.
\begin{theorem}\label{thm:doubling.lc}
Let $G$ be a locally compact group with Haar measure $\mu$, and let $K\ge1$. Suppose $S\subseteq G$ is a precompact symmetric open set containing the identity such that $\mu(S^{2n})\le K\mu(S^n)$ for some integer $n\ge8K^4$. Then $\mu(S^{3n})\le(2^{12}K^{24})^{10\cdot5^{4K^4}}\mu(S^n)$, and $S^{2n}$ is a precompact open $(2^{12}K^{24})^{30\cdot5^{4K^4}}$-approximate group.
\end{theorem}
Here and throughout, by a Haar measure we mean a \emph{left Haar} measure, so that $\mu(xA)=\mu(A)$ for all $A\subseteq G$ and $x\in G$.

A particular reason to be interested in locally compact groups is that the automorphism group of a vertex-transitive graph $\Gamma$ is a locally compact group with respect to the topology of pointwise convergence (see e.g. \cite[\S4]{ttTrof} for details). Given a vertex $o\in\Gamma$, the set $S=\{g\in\Aut(\Gamma):d(o,g(o))\le1\}$ is a compact open generating set for $\Aut(\Gamma)$; moreover, if we write $\beta_\Gamma(n)$ for the number of vertices in a ball of radius $n$ in $\Gamma$, and $\mu$ for the left Haar measure on $\Aut(\Gamma)$ normalised so that the stabiliser of $o$ has measure $1$, then we have $\beta_\Gamma(n)=\mu(S^n)$ for each $n\in\N$ \cite[Lemma 4.8]{ttTrof}. \cref{thm:doubling.lc} therefore gives rise to the following corollary.
\begin{corollary}\label{cor:trans}
Let $K\ge1$. Suppose $\Gamma$ is a locally finite vertex-transitive graph satisfying $\beta_\Gamma(2n)\le K\beta_\Gamma(n)$ for some integer $n\ge8K^4$. Then $\beta_\Gamma(3n)\le(2^{12}K^{24})^{10\cdot5^{4K^4}}\beta_\Gamma(n)$.
\end{corollary}
In the special case of discrete groups in \cref{thm:doubling.lc}, or equivalently Cayley graphs in \cref{cor:trans}, our argument gives slightly stronger bounds, as follows.
\begin{theorem}\label{thm:doubling=>tripling}
Let $K\ge1$. Suppose $S$ is a finite symmetric subset of a group containing the identity, and $|S^{2n}|\le K|S^n|$ for some integer $n\ge2K^2$. Then $|S^{3n}|\le K^3(3^9K^{18})^{10\cdot5^{K^2}}|S^n|$, and $S^{2n}$ is an $K^9(3^9K^{18})^{30\cdot5^{K^2}}$-approximate group.
\end{theorem}
We see no reason to believe that the bounds of \cref{thm:doubling.lc,thm:doubling=>tripling} are anything close to optimal.

To state the obvious, the above results can immediately be applied to weaken the assumptions required in any result about groups or transitive graphs with a tripling bound on a ball of large radius amongst its hypotheses. Such results include the following, for example:
\begin{itemize}
\item Breuillard and the second author \cite[Theorem 1.1]{bt} show that for all $K\ge1$ there exist $n_0=n_0(K)$ and $C=C(K)\ge1$ such that if $S$ is a symmetric subset of a group containing the identity, and $|S^{2n+1}|\le K|S^n|$ for some integer $n\ge n_0$, then $|S^{rm}|\le C^r|S^m|$ for all $r,m\in\N$ with $m\ge n$.

\item Easo and Hutchcroft \cite{unif.finite.pres} have recently proved, and we have refined in other work \cite{ttBalls}, a uniform version of the classical fact that a group of polynomial growth is finitely presented. To state this result formally, we borrow some terminology from Easo and Hutchcroft's paper. Let $G$ be a group with finite generating set $S$, so that $G\cong F_S / R$ for some normal subgroup $R$ of the free group $F_S$. For each $n \in\N$, let $R_n$ be the set of words of length at most $2^n$ in the free group $F_S$ that are equal to the identity in $G$, and let $\langle R_n \rangle^{F_S}$ be the normal subgroup of $F_S$ generated by $R_n$. Say that $(G,S)$ has a \emph{new relation on scale $n$} if $\langle R_{n+1} \rangle^{F_S}\ne\langle R_n \rangle^{F_S}$. We show that for each $K\ge1$ there exist $n_0=n_0(K)\in\N$ such that if $G$ is a group and $S$ is a finite symmetric generating set for $G$ containing the identity and satisfying $|S^{3n}|\le K|S^n|$ for some integer $n\ge n_0$ then
\[
\#\Bigl\{m\in \mathbb{N} : m\ge \log_2n \text{ and $(G,S)$ has a new relation on scale $m$} \Bigr\}\ll_{K}1
\]
(Easo and Hutchcroft's result was the same but with the upper bound depending on $|S|$ as well as $K$).

\item Trofimov \cite{trof} has famously shown that an arbitrary vertex-transitive graph $\G$ of polynomial growth is quasi-isometric to a virtually nilpotent Cayley graph. In a previous paper of ours \cite{ttTrof}, we proved a finitary refinement of Trofimov's theorem \cite[Theorem 2.3]{ttTrof} under the hypothesis $\beta_\G(3n)\le K\beta_\G(n)$.
\end{itemize}

Using either the non-abelian analogues of the Plünnecke--Ruzsa inequalities or the trivial observation that if $A$ is a $K$-approximate group then $|A^m|\le K^{m-1}|A|$, one can also easily bound the quantities $|S^{mn}|$, $\mu(S^{mn})$ or $\beta_\G(mn)$ for $m>3$ in the above results. In other work \cite{ttBalls} we provide (using a far more involved argument) bounds on higher powers that are much stronger for large values of $n$ (and in fact, in the case of \cref{thm:doubling=>tripling} the result of Breuillard and the second author listed just above already gives stronger bounds for large $n$).

It is natural to ask whether the hypotheses of \cref{thm:doubling=>tripling,thm:doubling.lc} can be weakened still further.
\begin{conjecture}
Let $K\ge1$ and $\eps>0$. Suppose $G$ is a locally compact group with Haar measure $\mu$, and $S\subseteq G$ is a precompact symmetric open set containing the identity such that $\mu(S^{\lceil(1+\eps)n\rceil})\le K\mu(S^n)$ for some integer $n$ that is sufficiently large depending on $K$ and $\eps$. Then $\mu(S^{3n})\ll_{K,\eps}\mu(S^n)$.
\end{conjecture}
Natural partial results to aim for would be to replace $\mu(S^{\lceil(1+\eps)n\rceil})\le K\mu(S^n)$ with something stronger like $\mu(S^{2n-1})\le K\mu(S^n)$, or to consider the special case of finitely generated $G$. Note that the hypothesis $\mu(S^{\lceil(1+\eps)n\rceil})\le K\mu(S^n)$ cannot be weakened to e.g. $\mu(S^{n+1})\le K\mu(S^n)$, since this is satisfied for example for all $n$ by any finite set $S$ of cardinality at most $K$. At this point we would like to acknowledge that Breuillard has asked questions along these lines in the past.


\section{Approximating sets of small doubling by approximate groups}
In order to prove \cref{thm:doubling.lc}, we make use of the following result, which was one of the main motivations for the study of approximate groups in the first place.
\begin{prop}[Tao \cite{tao.product.set}]\label{prop:doubling.covered.by.app.grp-lc}
Let $G$ be a locally compact group with Haar measure $\mu$, and let $K\ge1$. Suppose $A\subseteq G$ is a precompact symmetric open set containing the identity and satisfying $\mu(A^2)\le K\mu(A)$. Then there exists a precompact open $2^{12}K^{24}$-approximate group $U\subseteq A^4$ with measure $\mu(U)\le 4K^5\mu(A)$ and a finite set $X\subseteq A$ of cardinality at most $4K^4$ such that $A\subseteq XU$.
\end{prop}
At first sight, this is just a special case of \cite[Theorem 4.6]{tao.product.set} with bounds written more explicitly. However, in Tao's paper there is a global assumption that the Haar measure is bi-invariant, and in order to deduce \cref{cor:trans} from \cref{thm:doubling.lc} we need this result for arbitrary left-Haar measures, not just bi-invariant ones. The main aim of this section is therefore to prove \cref{prop:doubling.covered.by.app.grp-lc} for general Haar measures, essentially following Tao but restricting some of the ingredients to symmetric sets in order to avoid his bi-invariance assumption.

We remark that slightly better bounds are available in the discrete setting, as follows.
\begin{prop}[{\cite[Theorem 2.5.6]{book}}]\label{prop:doubling.covered.by.app.grp}
Let $K\ge1$, and suppose $A$ is a finite subset of a group satisfying $|A^2|\le K|A|$. Then there exists an $3^9K^{18}$-approximate group $U\subseteq A^2$ and a set $X\subseteq A$ of size at most $K^2$ such that $A\subseteq XU$.
\end{prop}
The bounds are not stated this explicitly in the reference, but can easily be read out of the proof (including the exponent $18$ in the approximate-group parameter, which is erroneously stated as $24$ in the reference).

To produce the approximate group $U$ required by \cref{prop:doubling.covered.by.app.grp-lc} we use the following result, which also has its origins Tao's paper \cite{tao.product.set}.
\begin{lemma}[{\cite[Proposition 6.1]{ttTrof}}]\label{lem:tripling->AG}
Let $G$ be a locally compact group with Haar measure $\mu$, and let $K\ge1$. Suppose $A\subseteq G$ is a precompact symmetric open set such that $\mu(A^3)\le K\mu(A)$. Then $A^2$ is a precompact open $K^3$-approximate group.
\end{lemma}
The technical heart of the proof of \cref{prop:doubling.covered.by.app.grp-lc} is then the following result, which is non-bi-invariant version of \cite[Proposition 4.5]{tao.product.set}.
\begin{prop}\label{prop:tao.4.5}
Let $G$ be a locally compact group with Haar measure $\mu$, and let $K\ge1$. Suppose $A\subseteq G$ is a non-empty precompact symmetric open set satisfying $\mu(A^2)\le K\mu(A)$. Then there exists a precompact symmetric open set $V\subseteq A^2$ containing the identity such that $\mu(V)\ge\mu(A)/2K$ and $\mu(AV^nA)\le2^nK^{2n+1}\mu(A)$ for all $n\in\N$.
\end{prop}
Before we begin the proof of \cref{prop:tao.4.5}, let us recall some notation and basic facts from Tao's paper. Given two absolutely integrable functions $f,g:G\to\R$, we define as usual the \emph{convolution} $f\ast g$ via
\[
f\ast g(x)=\int_Gf(y)g(y^{-1}x)\,d\mu(y).
\]
Given two non-empty open precompact sets $A,B\subseteq G$, we then define the \emph{multiplicative energy} $E(A,B)$ of $A$ and $B$ via
\[
E(A,B)=\int_G1_A\ast1_B(x)^2\,d\mu(x).
\]
By Fubini's theorem and a change of variables we have
\begin{equation}\label{eq:tao6}
\int_G1_A\ast1_B(x)\,d\mu(x)=\mu(A)\mu(B).
\end{equation}
Since $1_A\ast1_B$ is supported on $AB$, combining this with Cauchy--Schwarz then gives
\begin{equation}\label{eq:tao9}
E(A,B)\ge\frac{\mu(A)^2\mu(B)^2}{\mu(AB)}.
\end{equation}
\begin{lemma}\label{lem:1A*1B}
Let $G$ be a locally compact group with Haar measure $\mu$, and suppose $A,B\subseteq G$ are non-empty precompact open sets, with $B$ symmetric. Then $1_A\ast1_B(x)=\mu(A\cap xB)$ for all $x\in G$.
\end{lemma}
\begin{proof}
Given $x\in G$ we have
\begin{align*}
1_A\ast1_B(x)&=\int_G1_A(y)1_B(y^{-1}x)\,d\mu(y)\\
		&=\int_G1_A(y)1_B(x^{-1}y)\,d\mu(y)\\
		&=\int_G1_A(y)1_{xB}(y)\,d\mu(y)\\
		&=\mu(A\cap xB),
\end{align*}
as required.
\end{proof}
\begin{lemma}\label{lem:1A*1B.continuous}
Let $G$ be a locally compact group with Haar measure $\mu$, and suppose $A,B\subseteq G$ are measurable sets of finite measure. Then the map $G\to\R$, $x\mapsto 1_A\ast1_B(x)$ is continuous.
\end{lemma}
\begin{proof}
Observe that $1_A\ast1_B(x)=\langle 1_A,\lambda(x^{-1})1_B\rangle$, where $\lambda$ denotes the left regular representation in $L^2(G)$. Hence the statement follows from the continuity of $\lambda$.
\end{proof}
\begin{proof}[Proof of \cref{prop:tao.4.5}]
Following Tao, define
\[
V=\{x\in G:\mu(A\cap xA)>\mu(A)/2K\},
\]
noting that $V\subseteq A^2$ and contains the identity, and is open by \cref{lem:1A*1B,lem:1A*1B.continuous}, symmetric because $\mu(x^{-1}A\cap A)=\mu(A\cap xA)$, and precompact because $V\subseteq A^2$. \cref{lem:1A*1B} and \eqref{eq:tao6} imply that
\[
\int_G\mu(A\cap xA)\,d\mu(x)=\mu(A)^2,
\]
so that
\[
\int_{G\setminus V}\mu(A\cap xA)^2\,d\mu(x)\le\int_G\frac{\mu(A)}{2K}\mu(A\cap xA)\,d\mu(x)\le\frac{\mu(A)^3}{2K}.
\]
On the other hand, we have
\begin{align*}
\int_G\mu(A\cap xA)^2\,d\mu(x)&=\int_G1_A\ast1_A(x)^2\,d\mu(x)&\text{(by \cref{lem:1A*1B})}\\
				&=E(A,A)\\
				&\ge\mu(A)^4/\mu(A^2)&\text{(by \eqref{eq:tao9})}\\
				&\ge\mu(A)^3/K,
\end{align*}
so that
\[
\int_V\mu(A\cap xA)^2\,d\mu(x)\ge\frac{\mu(A)^3}{2K}.
\]
Since $\mu(A\cap xA)\le\mu(A)$, this implies in particular that $\mu(V)\ge\mu(A)/2K$, as required.

Continuing to follow Tao, let $n\in\N$ and $x\in AV^nA$, and write $x=a_0v_1\cdots v_nb_{n+1}$ with $a_0,b_{n+1}\in A$ and each $v_i\in V$. Making the successive changes of variables
\begin{align*}
y_0 &= a_0 b_1\\
y_i &= b_i^{-1} v_i b_{i+1}&(i=1,\ldots,n-1),
\end{align*}
we may write
\[
\int_{(A^2)^n}1_{A^2}(y_{n-1}^{-1}\cdots y_0^{-1}x)\,d\mu(y_0)\cdots d\mu(y_{n-1})=\int_{G^n}1_{A^2}(a_0b_1)\prod_{i=1}^n1_{A^2}(b_i^{-1} v_i b_{i+1})\,d\mu(b_1)\cdots d\mu(b_n).
\]
By symmetry of $A$, if $b_1,\ldots,b_n\in A$ and $v_1^{-1}b_1,\ldots,v_n^{-1}b_n\in A$ then the integrand on the right-hand side is equal to $1$, so we may bound
\begin{align*}
\int_{(A^2)^n}1_{A^2}(y_{n-1}^{-1}\cdots y_0^{-1}x)\,d\mu(y_0)\cdots d\mu(y_{n-1})&\ge\int_{G^n}\prod_{i=1}^n1_{A}(b_i)1_A(v_i^{-1}b_i)\,d\mu(b_1)\cdots d\mu(b_n)\\
	&=\prod_{i=1}^n\mu(A\cap v_iA)\\
	&>\mu(A)^n/(2K)^n,
\end{align*}
the last inequality following from the definition of $V$. Note that $x=y_0\ldots y_n$. Integrating over all $x$ and changing variables, we conclude that
\begin{align*}
\mu(AV^nA)\mu(A)^n/(2K)^n&<\int_{AV^nA}\int_{(A^2)^n}1_{A^2}(y_{n-1}^{-1}\cdots y_0^{-1}x)\,d\mu(y_0)\cdots d\mu(y_{n-1})d\mu(x)\\
			&=\int_{(A^2)^{n+1}}1_{AV^nA}(y_0\cdots y_n)\,d\mu(y_0)\cdots d\mu(y_n)\\
			&\le\mu(A^2)^{n+1}\\
			&\le K^{n+1}\mu(A)^{n+1},
\end{align*}
so that $\mu(AV^nA)\le2^nK^{2n+1}\mu(A)$ as required.
\end{proof}
\begin{proof}[Proof of \cref{prop:doubling.covered.by.app.grp-lc}]
By \cref{prop:tao.4.5} there exists a precompact symmetric open set $V\subseteq A^2$ such that $\mu(V)\ge\mu(A)/2K$ and, since $A$ contains the identity, $\mu(V^2)\le\mu(AV^2A)\le4K^5\mu(A)$ and $\mu(V^3)\le\mu(AV^3A)\le8K^7\mu(A)\le16K^8\mu(V)$. \cref{lem:tripling->AG} then implies that $V^2$ is a precompact open $2^{12}K^{24}$-approximate group. The application of \cref{prop:tao.4.5} also gives $\mu(AV)\le\mu(AVA)\le2K^3\mu(A)\le4K^4\mu(V)$, so that Ruzsa's covering argument \cite{ruzsa} as presented for example in \cite[Lemma 2.4.4]{book} implies that there exists $X\subseteq A$ of cardinality at most $4K^4$ such that $A\subseteq XV^2$. Take $U=V^2$.
\end{proof}

\section{Small doubling implies small tripling}
The basic idea of the proof of \cref{thm:doubling.lc} is to show that the growth of $S^n$ can be controlled in terms of the growth of the approximate group $U$ given by \cref{prop:doubling.covered.by.app.grp-lc}. We first refine $U$ so that translates of it by distinct elements of $X$ are strongly disjoint.
\begin{lemma}\label{eq:xU.disjoint}
Suppose $X$ and $U$ are subsets of a group, with $U$ symmetric and containing the identity and $X$ finite. Then there exist $m\le5^{|X|-1}$ and $X'\subseteq X$ such that $x\notin yU^{4m}$ for all distinct $x,y\in X'$, and such that $XU\subseteq X'U^m$.
\end{lemma}
\begin{proof}
If $x\notin yU^{4}$ for all distinct $x,y\in X$, which in particular holds vacuously if $|X|=1$, then we may take $m=1$ and $X'=X$. We may therefore assume that there exist distinct $x,y\in X$ such that $x\in yU^{4}$, and by induction that the lemma holds for all smaller sets $X$. The first of these assumptions implies in particular that $XU\subseteq(X\setminus\{x\})U^5$, and the induction hypothesis then implies that there exists $k\le5^{|X|-2}$ and $X'\subseteq X\setminus\{x\}$ such that $y\notin zU^{20k}$ for all distinct $y,z\in X'$, and such that $XU\subseteq(X\setminus\{x\})U^5\subseteq X'U^{5k}$. We may therefore take this $X'$, and $m=5k$.
\end{proof}
One can think of \cref{eq:xU.disjoint} as saying that the translates $xU^m$ of the approximate group $U^m$ by elements $x\in X'$ behave locally like cosets of a genuine subgroup. More precisely, given a subgroup $H$ of a group $G$ and elements $x,y\in G$, we trivially have $y\in xH$ if and only if $yH=xH$. The approximate group $U^m$ given by \cref{eq:xU.disjoint} satisfies a `local' version of this property, in that for all $x\in X'$ and $y\in X'U^m$, we have $y\in xU^m$ if and only if $yU^m\subseteq xU^{2m}$ and $xU^m\subseteq yU^{2m}$. The `only if' direction of this is trivial. To prove the `if' direction, first note that if $yU^m\subseteq xU^{2m}$ then
\begin{equation}\label{eq:y.in.xU}
y\in xU^{3m}.
\end{equation}
By assumption there exists $z\in X'$ such that $y\in zU^m$, but then combined with \eqref{eq:y.in.xU} and the conclusion of \cref{eq:xU.disjoint} this easily implies that $z=x$ and hence $y\in xU^m$.

The utility of this is that it allows us to prove a version of a rather useful property of the cosets of a genuine subgroup $H$ of a group $G$ with finite symmetric generating set $S$, namely that if $H$ has index bounded by $k\in\N$ then one may find a complete set of coset representatives for $H$ inside $S^{k-1}$. This observation is crucial in Breuillard, Green and Tao's finitary refinement of Gromov's theorem \cite[Corollary 11.2]{bgt}, for example. In our next result, which is the key new insight of the present work, we prove the following version of this in the setting in which we have a set of translates of a set $U$ that are strongly disjoint in the sense of \cref{eq:xU.disjoint}.
\begin{prop}\label{prop:translates.app.grp.bdd.reps}
Let $k,n\in\N$. Suppose that $S$ and $U$ are symmetric subsets of a group, each containing the identity, and that $X$ is a subset of size at most $k$ such that $S^n\subseteq XU$ and $x\notin yU^4$ for all distinct $x,y\in X$. Then there exists $X'\subseteq S^{k-1}$ with $|X'|\le|X|$ such that $S^n\subseteq X'U^2$.
\end{prop}
\begin{proof}
We may assume that $X$ is minimal such that $S^n\subseteq XU$. For each $r=0,1,\ldots,n$, define $X_r=\{x\in X:xU\cap S^r\ne\varnothing\}$, noting that $X_n=X$ by minimality of $X$.

We claim that if $X_{r+1}=X_r$ for some $r\le n-2$ then $X_{r+2}=X_{r+1}$. To see this, suppose that $X_{r+1}=X_r$, and let $x\in X_{r+2}$. By definition, this means that there exists $u\in U$ such that $xu\in S^{r+2}$. This in turn implies that there exists $s\in S$ such that $xu\in sS^{r+1}\subseteq sX_{r+1}U=sX_rU$. It follows that there exist $y\in X_r$ and $v\in U$ such that $xu=syv$. By definition of $X_r$, there exists $w\in U$ such that $yw\in S^r$. It then follows that $syw\in S^{r+1}$, hence $xu=syv\in S^{r+1}U^2$, and hence $x\in S^{r+1}U^3\subseteq X_{r+1}U^4$. By hypothesis, it then follows that $x\in X_{r+1}$ as claimed.

By induction, this claim implies that if $X_{r+1}=X_r$ for some $r< n$ then $X_r=X_n=X$. Since $|X_0|\ge1$, it follows that if $k\le n$ then $X_{k-1}=X$. Combined with the fact that $X_n=X$, this implies in all cases that for each $x\in X$ there exists $x'\in S^{k-1}$ such that $x'\in xU$, and hence $xU\subseteq x'U^2$ by symmetry of $U$. Setting $X'=\{x':x\in X\}$, we therefore have $X'\subseteq S^{k-1}$, $|X'|\le|X|$ and $S^n\subseteq XU\subseteq X'U^2$ as required.
\end{proof}

The next lemma shows that if $X$ is contained in a ball of small enough radius in $S$ and $S^n\subseteq XU$ then the further growth of $S$ is controlled by the growth of $U$.

\begin{lemma}\label{lem:inclusions.local}
Let $k,r\in\N$. Suppose $G$ is a group with symmetric generating set $S$ containing the identity, and $X\subseteq S^k$ and $U\subseteq G$ satisfy $S^{r+k}\subseteq XU$. Then $S^{mr+k}\subseteq XU^m$ for all $m\in\N$.
\end{lemma}
\begin{proof}
The case $m=1$ is true by hypothesis, and for $m>1$ by induction we have $S^{mr+k}=S^rS^{(m-1)r+k}\subseteq S^rXU^{m-1}\subseteq S^{r+k}U^{m-1}\subseteq XU^m$, as required.
\end{proof}

\begin{proof}[Proof of \cref{thm:doubling=>tripling,thm:doubling.lc}]
We will prove \cref{thm:doubling.lc} in detail; the proof of \cref{thm:doubling=>tripling} is identical, but with \cref{prop:doubling.covered.by.app.grp} in place of \cref{prop:doubling.covered.by.app.grp-lc} right at the start. \cref{prop:doubling.covered.by.app.grp-lc} implies that there exists a $2^{12}K^{24}$-approximate group $U\subseteq S^{4n}$ satisfying $\mu(U)\le4K^5\mu(S^n)$ and a set $X\subseteq S^n$ of size at most $4K^4$ such that $S^n\subseteq XU$. \cref{eq:xU.disjoint} then implies that there exist $m\le5^{4K^4}$ and $X'\subseteq X$ such that $x\notin yU^{4m}$ for all distinct $x,y\in X'$, and such that $S^n\subseteq X'U^m$. \cref{prop:translates.app.grp.bdd.reps} then implies that there exists $X''\subseteq S^{4K^4-1}\subseteq S^{\lfloor n/2\rfloor}$ with $|X''|\le4K^4$ such that $S^n\subseteq X''U^{2\cdot5^{4K^4}}$. Applying \cref{lem:inclusions.local} with $k=\lfloor n/2\rfloor$ and $r=\lceil n/2\rceil$ then implies that $S^{3n}\subseteq X''U^{10\cdot5^{4K^4}}$, and hence that $\mu(S^{3n})\le4K^4\mu(U^{10\cdot5^{4K^4}})\le4K^4(2^{12}K^{24})^{10\cdot5^{4K^4}-1}\mu(U)\le16K^9(2^{12}K^{24})^{10\cdot5^{4K^4}-1}\mu(S^n)\le(2^{12}K^{24})^{10\cdot5^{4K^4}}\mu(S^n)$. The fact that $S^{2n}$ is a $(2^{12}K^{24})^{30\cdot5^{4K^4}}$-approximate group then follows from \cref{lem:tripling->AG}.
\end{proof}

\section*{Acknowledgments} 
We are grateful to four anonymous referees for comments on earlier drafts.

\bibliographystyle{amsplain}


\begin{dajauthors}
\begin{authorinfo}[romain]
  Romain Tessera\\
  Institut de Math\'ematiques de Jussieu-Paris Rive Gauche\\
  Universit\'e Paris Cit\'e \\
  France\\
  romain.tessera\imageat{}imj-prg\imagedot{}fr \\
  \url{https://www.normalesup.org/~tessera/}
\end{authorinfo}
\begin{authorinfo}[matt]
  Matthew Tointon\\
  School of Mathematics\\
  University of Bristol\\
  United Kingdom\\
  m.tointon\imageat{}bristol\imagedot{}ac\imagedot{}uk \\
  \url{https://www.tointon.neocities.org}
\end{authorinfo}
\end{dajauthors}

\end{document}